\documentclass[11pt]{amsart}

\usepackage{amsmath,amssymb,amsthm}
\usepackage[hidelinks]{hyperref}

\newcommand{\bit}[1]{\mathrm{bit}_{#1}}
\newcommand{\sd}{s_2}
\newcommand{\Mtop}{M}
\newcommand{\tcount}{T}
\newcommand{\bcount}{B}
\newcommand{\pcx}{p}          % subword complexity p_w(n)
\newcommand{\bwC}{C}          % Baker--Wuestholz constant C(n,d)
\newcommand{\mh}{h'}          % modified height
\newcommand{\N}{\mathbb{N}}

\newcommand{\Z}{\mathbb{Z}}
\newcommand{\Q}{\mathbb{Q}}

\theoremstyle{plain}
\newtheorem{theorem}{Theorem}[section]
\newtheorem{theoremintro}{Theorem}

\newtheorem{proposition}[theorem]{Proposition}
\newtheorem{lemma}[theorem]{Lemma}

\theoremstyle{definition}
\newtheorem{definition}[theorem]{Definition}

\theoremstyle{remark}
\newtheorem{remark}[theorem]{Remark}

\numberwithin{equation}{section}

\begin{document}

\title[Aperiodicity and subword complexity of $3^m$ in binary]{Aperiodicity and
subword complexity\\ in the binary expansion of powers of three}
\author{Ralf Stephan}
\thanks{Institute for Globally Distributed Open Research and Education (IGDORE)}

\date{\today}

\begin{abstract}
We prove two results on the fine structure of the binary digits of $3^{m}$.
First, for every fixed period $p$, the number of positions at which the binary
expansion of $3^{m}$ breaks $p$-periodicity grows in order like
$\log m/\log\log m$; equivalently, no window of the expansion deeper than a fixed
power of $\log m$ is $p$-periodic. Second, the finite binary word formed by the
low-order digits of $3^{m}$ meets the Morse--Hedlund floor: its complexity
function satisfies $\pcx_{3^{m}}(n)\ge n+1$ for every length $n$, once $m$ is
large enough.
\end{abstract}

\maketitle

%======================================================================
\section{Introduction}\label{sec:intro}
%======================================================================

Write a positive integer $N$ in binary, $N=\sum_{i\ge 0}\bit{i}(N)\,2^{i}$ with
$\bit{i}(N)\in\{0,1\}$, and let
\[
\Mtop(N):=\lfloor\log_{2}N\rfloor,\qquad
\sd(N):=\sum_{i\ge 0}\bit{i}(N),
\]
so that $N$ has $\Mtop(N)+1$ binary digits and $\sd(N)$ is its number of $1$-bits.
This note concerns the arrangement of the binary digits of $N=3^{m}$. Throughout,
``$\log$'' is the natural logarithm.

We use two further statistics of the digit string. A \emph{digit change} of $N$
is an index $i<\Mtop(N)$ with $\bit{i}(N)\neq\bit{i+1}(N)$; writing $N$ as a
concatenation of maximal runs of equal digits (its \emph{blocks}), the number of
digit changes is one less than the number of blocks. For a fixed period
$p\ge 1$, a \emph{period-$p$ break} of $N$ is an index $i<\Mtop(N)$ with
$\bit{i}(N)\neq\bit{i+p}(N)$.

\begin{definition}\label{def:transition}
For $N\ge 1$, let
\begin{gather*}
\tcount(N):=\#\{\,i<\Mtop(N):\bit{i}(N)\neq\bit{i+1}(N)\,\},\\
\bcount_{p}(N):=\#\{\,i<\Mtop(N):\bit{i}(N)\neq\bit{i+p}(N)\,\}.
\end{gather*}
Thus $\tcount(N)$ is the number of digit changes of $N$ (its number of blocks is
$\tcount(N)+1$), and $\bcount_{p}(N)$ is its number of period-$p$ breaks; note
$\bcount_{1}=\tcount$.
\end{definition}

Finally, for a word $w=(w_{0},w_{1},\dots)$ over a finite alphabet, the
\emph{subword complexity} (or \emph{factor complexity}) $\pcx_{w}(n)$ is the
number of distinct factors of length $n$ occurring in $w$. We read the low bits
of $3^{m}$ as the finite word
$w=(\bit{0}(3^{m}),\dots,\bit{\Mtop-1}(3^{m}))$, $\Mtop=\Mtop(3^{m})$, and write
$\pcx_{3^{m}}(n)$ for its complexity: the number of distinct length-$n$ factors at
the positions $0,1,\dots,\Mtop-n$.

\subsection{Known results}\label{sub:known}

The starting point is a theorem of Stewart on integers with few digits in two
bases. In the notation of \cite{Ste80}, $\sd(N)$ is the number of nonzero binary
digits, and Stewart's Theorem~2, applied to the recurrence $u_{n}=3^{n}$ with
base $2$, gives the following.

\begin{theorem}[Stewart \cite{Ste80}]\label{thm:stewart}
There is a constant $C>0$ such that
\[
\sd(3^{m})\ \ge\ \frac{\log m}{\log\log m+C}-1
\qquad\text{for every } m\ge 2 .
\]
In particular $\sd(3^{m})\to\infty$: no power of three eventually has sparse
binary digits.
\end{theorem}

Stewart's proof rests on Baker's theory of linear forms in logarithms; in the
present work the corresponding input is the following effective theorem of Baker
and Wüstholz \cite{BW93}, stated in Section~\ref{sec:tools}. For an embedding
$\varphi\colon K\to\mathbb{C}$ of a number field $K$ of degree $d$, write
$\mh(\alpha)$ for the modified height of $\alpha\in K$ and $\bwC(n,d)$ for the
Baker--Wüstholz constant; the theorem asserts that a nonzero linear form
$\Lambda=\sum_{i}b_{i}\log\varphi(\alpha_{i})$ with integer coefficients bounded
by $B\ge 2$ obeys $\log|\Lambda|\ge-\bwC(n,d)\max(\log B,\tfrac1d)\prod_{i}
\mh(\alpha_{i})$ (Theorem~\ref{thm:bw}).

Counting the $1$-bits says nothing about their placement, and it is natural to
count instead the blocks of $3^{m}$. That the number of blocks tends to infinity
is a theorem of Blecksmith, Filaseta and Nicol \cite{BFN93}, who introduced this
quantity; an effective, and faster, lower bound follows from work of Bugeaud and
Kaneko. Since $3^{m}$ is an integral $S$-unit with $S=\{3\}$, their Corollary~1.5
and the block variant recorded in their Remark~4.4 apply.

\begin{theorem}[Blecksmith--Filaseta--Nicol; Bugeaud--Kaneko]\label{thm:blocks}
The number of binary blocks of $3^{m}$ tends to infinity, and grows in order at
least like $\log m/\log\log m$; equivalently
\[
\tcount(3^{m})\ \ge\ \frac{\log m}{\log\log m+C}-1
\]
for a constant $C>0$ and all $m\ge 2$.
\end{theorem}

Here the order $\log m/\log\log m$ comes from
$\log\log(3^{m})=\log m+O(1)$ in \cite[Cor.~1.5, Rem.~4.4]{BK17}. On the opposite,
sparse side, the powers of three with few nonzero bits are known exactly in the
relevant range. The case $\sd\le 2$ is elementary; the case $\sd=3$ is a theorem
of Dimitrov and Howe.

\begin{theorem}[Dimitrov--Howe \cite{DH23}]\label{thm:DH}
The equation $3^{m}=2^{a}+2^{b}+1$ with $1\le b<a$ has the unique solution
$(m,b,a)=(4,4,6)$. More generally \cite[Thm.~1.1]{DH23}, the powers of three with
at most $22$ nonzero binary digits are exactly $3^{0},\dots,3^{25}$.
\end{theorem}

\begin{theorem}[sparse powers of three]\label{thm:sparse}
$\sd(3^{m})\le 2$ if and only if $m\le 2$, and $\sd(3^{m})=3$ if and only if
$m=4$. Explicitly $3^{0}=1_{2}$, $3^{1}=11_{2}$, $3^{2}=1001_{2}$ have
$\sd\le 2$, and $3^{4}=1010001_{2}$ is the unique power of three with
$\sd=3$.
\end{theorem}

The first assertion of Theorem~\ref{thm:sparse} is elementary (peeling the least
significant bit reduces it to ruling out $3^{m}=2^{a}+1$ for $m\ge 3$); the second
is the $\sd=3$ reading of Theorem~\ref{thm:DH}. The smallest powers of three
breaking the sparse bound are $3^{3}=11011_{2}$ ($\sd=4$) and
$3^{4}=1010001_{2}$ ($\sd=3$).

The combinatorial input to our second result is a finite form of the
Morse--Hedlund theorem \cite{MH38}: an infinite word of subword complexity
$\pcx_{w}(n)\le n$ for some $n$ is eventually periodic. We use a quantitative
finite version, stated and proved in Section~\ref{sec:complexity}
(Lemma~\ref{lem:fmh}), in the spirit of the special-factor analysis of Carpi and
de Luca \cite{CdL00}.

\subsection{Results of this paper}\label{sub:new}

Theorems~\ref{thm:stewart} and~\ref{thm:blocks} forbid long constant runs. We
strengthen this in two directions, neither of which had, at the time of writing,
been treated for the integer $3^{m}$; the existing quantitative results on digit
changes and on subword complexity in this circle concern algebraic irrationals,
whose expansions are infinite, rather than the finite word of a single power.
Section~\ref{sub:added} records a subsequent development, which supersedes the
qualitative content of Theorem~\ref{thm:complexity} while leaving its effective
form, and Theorem~\ref{thm:aperiodic}, untouched.

Our first result forbids long \emph{periodic} runs, of any fixed period. It is
proved in Section~\ref{sec:aperiodic}.

\begin{theoremintro}\label{thm:aperiodic}
For every fixed period $p\ge 1$ there is a constant $C_{p}>0$ such that
\[
\bcount_{p}(3^{m})\ \ge\ \frac{\log m}{\log\log m+C_{p}}-2
\qquad\text{for every } m\ge 2;
\]
in particular $\bcount_{p}(3^{m})\to\infty$. Equivalently, no window of the
binary expansion of $3^{m}$ deeper than a fixed power of $\log m$ can be
$p$-periodic.
\end{theoremintro}

The case $p=1$ recovers Theorem~\ref{thm:blocks}. The engine is a new estimate
for periodic windows (Lemma~\ref{lem:gapP}), a four-logarithm companion of the
run estimates behind the classical bounds. The constant is explicit: one may
take $C_{1}<39.31$, and $C_{p}<40+3\log(p+1)$ for every $p\ge 1$
(Remark~\ref{rem:constA}).

Our second result is a Morse--Hedlund floor for the finite binary word of a
single $3^{m}$. It is proved in Section~\ref{sec:complexity}.

\begin{theoremintro}\label{thm:complexity}
For every length $n\ge 1$, the low-order binary word of $3^{m}$ meets the
Morse--Hedlund floor,
\[
\pcx_{3^{m}}(n)\ \ge\ n+1,
\]
for all sufficiently large $m$.
\end{theoremintro}

The value $n+1$ is the aperiodicity threshold, not the maximum available to a
binary word; Remark~\ref{rem:calibration} states precisely what the theorem does
and does not assert.

The threshold in $m$ is effective: $\pcx_{3^{m}}(n)\ge n+1$ once $m$ exceeds the point
where the linear growth of $\Mtop(3^{m})\ge m$ overtakes a bound of order
$n^{2}\log m$ coming from Theorem~\ref{thm:aperiodic}'s engine. Numerically it
suffices to take $m\ge 2\times 10^{18}n^{2}(1+\log n)$, and
$m\ge 1.33\times 10^{18}$ when $n=1$ (Remark~\ref{rem:constB}).

Neither theorem hides an ineffective constant. The single quantitative input is
the Baker--W\"ustholz constant $\bwC(4,1)$, which \cite{BW93} specifies by a
closed formula; Remark~\ref{rem:explicit} evaluates it and traces it through to
the constants above. We do not claim these numbers are of the right order of
magnitude---they are not---only that the results are effective rather than
qualitative.

Both theorems have been formally verified in the Lean~4 proof assistant;
Appendix~\ref{app:lean} records the Lean name and status of every statement
appearing in this paper.

\subsection{Note added}\label{sub:added}

After the present paper was posted, Bugeaud \cite{Bug26} took up the same
question by a different route and reached, ineffectively, a conclusion beyond
the reach of the method used here. The comparison divides cleanly, and we record
it.

Reading the base-$b$ representation of an integer as a finite word, as we do,
\cite[Thm.~1.2]{Bug26} proves the following. Let $b\ge 2$, let $C\ge 2$, and let
$S$ be a finite set of primes. If $x$ has all its prime factors in $S$, if its
base-$b$ word $a_{0}\dots a_{n}$ satisfies $b^{\lfloor n/(C+2)\rfloor}\nmid x$,
and if $x$ is large enough in terms of $b$, $S$ and $C$, then that word has more
than $C\ell$ distinct factors of length $\ell:=\lceil n/(C+2)\rceil$; the proof
in fact yields that its $n+2-\ell$ factors of that length are pairwise distinct.
Taking $b=2$ and $S=\{3\}$ the divisibility hypothesis is vacuous, so the binary
word of $3^{m}$ has factor complexity above any prescribed linear function once
$m$ is large. In particular it is not Sturmian \cite[Thm.~1.1]{Bug26}, and it is
not generated by a $k$-uniform morphism on an alphabet of $D$ letters with $k+D$
bounded \cite[Cor.~1.3]{Bug26}. (The word of \cite{Bug26} is the reverse of ours
with the leading bit restored; factor complexity is invariant under reversal and
moves by at most one under deletion of a letter, so the two readings agree here.)

\emph{Qualitatively, \cite{Bug26} is stronger, and subsumes
Theorem~\ref{thm:complexity}.} Sturmian words are exactly those meeting the floor
$\pcx(n)=n+1$ at every $n$, so Theorem~\ref{thm:complexity} does not exclude
them, whereas \cite[Thm.~1.1]{Bug26} does; this is precisely the strengthening
that Remark~\ref{rem:calibration} identifies as unreachable from
Lemma~\ref{lem:gapP}. Moreover \cite[Thm.~1.2]{Bug26} implies
Theorem~\ref{thm:complexity} qualitatively, by way of our
Lemma~\ref{lem:fmh}: if $\pcx_{3^{m}}(n)\le n$ for a fixed $n$ and
$\Mtop\ge 3n$, that lemma supplies a period $p\le n$ valid on a factor of length
$\Mtop-2n$ beginning at some $a\le 2n$, so that at any length $\ell\le\Mtop-2n$
all but at most $4n$ of the starting positions carry a factor determined by its
residue modulo $p$, giving $\pcx_{3^{m}}(\ell)\le 5n$; taking $C=2$ and
$\ell=\lceil\Mtop/4\rceil$ contradicts $\pcx_{3^{m}}(\ell)\ge\tfrac34\Mtop$ for
$\Mtop$ large.

\emph{Quantitatively, nothing above is superseded.} Theorems 1.1--1.4 of
\cite{Bug26} rest on the $p$-adic Subspace theorem of Schmidt and Schlickewei,
which is ineffective; \cite{Bug26} says so explicitly, and no threshold in $m$
can be extracted from those proofs. Theorem~\ref{thm:complexity} carries the
explicit $m_{0}(n)$ of Remark~\ref{rem:constB}, and Theorem~\ref{thm:aperiodic}
the explicit rate $\log m/(\log\log m+C_{p})$ with $C_{p}<40+3\log(p+1)$ for the
number of period-$p$ breaks. For the latter \cite{Bug26} yields no rate at all:
its Theorem~1.2 gives $\bcount_{p}(3^{m})\ge C$ only for $m$ beyond an
uncomputable bound depending on $C$.

\emph{One effective statement of \cite{Bug26} does go past
Theorem~\ref{thm:aperiodic}}, on the narrower question of exact powers. Its
Theorem~1.6, proved with Matveev's estimate \cite{Mat00}, shows that if the
base-$b$ word of an integer with all prime factors in $S$ is a power $W^{t}$,
then $t\le C(S,b)$ with $C(S,b)$ effectively computable; at $b=2$, $S=\{3\}$
this bounds the exponent of a periodic binary word of $3^{m}$ by an absolute
constant. Theorem~\ref{thm:aperiodic} gives less here. If $a_{0}\dots a_{\Mtop}$
is a power $W^{t}$ with $\ell:=|W|$, then Lemma~\ref{lem:gapP} applied to the
window $x=1$, $y=\Mtop$ reads
\[
(\Mtop-1)\log 2\ \le\ 2\,\bwC(4,1)\,\ell^{2}\log 3\,
\max\bigl(\log(2m+\ell),1\bigr)+(\ell+1)\log 2 ,
\]
so that $\ell$ is bounded below only by a quantity of order
$\sqrt{m/\log m}$ and $t=(\Mtop+1)/\ell$ above only by one of order
$\sqrt{m\log m}$---growing with $m$, where \cite{Bug26} has a constant. What
Theorem~\ref{thm:aperiodic} supplies and \cite{Bug26} does not is the
window statement: a period-$p$ stretch anywhere in the expansion, not merely a
prefix, is bounded in depth, with a count of breaks growing in $m$.

\emph{The mechanism behind the difference} is worth isolating, since
Remark~\ref{rem:calibration} locates the $n+1$ ceiling in the combinatorial
half. Lemma~\ref{lem:fmh} needs the hypothesis $\pcx(n)\le n$ to produce a
period, and a period is what Lemma~\ref{lem:gapP} consumes. The argument of
\cite{Bug26} never passes through periodicity: from a low-complexity word it
extracts only a \emph{repetition}---two occurrences of one factor, at an
arbitrary offset---and converts that into a three-term relation
\cite[(2.2)]{Bug26} between $b^{s}x$, $x$ and a third integer, an analogue of
our \eqref{eq:key}, which is then fed to the Subspace theorem rather than to
Baker's theory. A repetition is available under the far weaker hypothesis
$\pcx(\ell)\le C\ell$, so the ceiling disappears; the price is the loss of
effectivity, and with it of everything Sections~\ref{sec:aperiodic}
and~\ref{sec:complexity} compute.

%======================================================================
\section{Linear forms in logarithms and the periodic-run estimate}\label{sec:tools}
%======================================================================

The engine common to the classical bounds and to Theorems~\ref{thm:aperiodic}
and~\ref{thm:complexity} is Stewart's: a long run, or more generally a long
periodic stretch, in the binary expansion of $3^{m}$ produces a small nonzero
linear form in the logarithms of $2$, $3$ and an auxiliary integer, which the
Baker--Wüstholz theorem forbids from being too small.

\begin{theorem}[Baker--Wüstholz \cite{BW93}]\label{thm:bw}
Let $\alpha_{1},\dots,\alpha_{n}$ be nonzero elements of a number field $K$ of
degree $d$, and let $b_{1},\dots,b_{n}\in\Z$ satisfy $|b_{i}|\le B$ with $B\ge 2$.
If $\Lambda:=\sum_{i}b_{i}\log\varphi(\alpha_{i})\neq 0$ $($principal branch$)$,
then
\[
\log|\Lambda|\ \ge\ -\,\bwC(n,d)\cdot\max\!\Bigl(\log B,\tfrac1d\Bigr)
\cdot\prod_{i}\mh(\alpha_{i}).
\]
\end{theorem}

We use Theorem~\ref{thm:bw} only over $K=\Q$ ($d=1$), where each $\alpha_{i}$ is
a positive rational and the modified height of a positive integer $a$ is
$\mh(a)=\max(\log a,1)$; thus $\mh(3)=\log 3$, while $\mh(2)=1$ because the
floor $1/d=1$ is active for $a=2$ (note $\log 2<1$). Set
$\Mtop:=\Mtop(3^{m})=\lfloor\log_{2}3^{m}\rfloor$.

\begin{remark}[All constants are explicit]\label{rem:explicit}
Theorem~\ref{thm:bw} is effective, and the constant $\bwC(n,d)$ occurring in it
is not an unspecified quantity: \cite{BW93} gives
\[
\bwC(n,d)\ =\ 18\,(n+1)!\;n^{\,n+1}\,(32d)^{\,n+2}\,\log(2nd).
\]
Every application in this paper is at $(n,d)=(4,1)$, so the only constant
entering the argument is the single number
\[
\bwC(4,1)\ =\ 18\cdot 5!\cdot 4^{5}\cdot 32^{6}\cdot\log 8
\ =\ 6480\cdot 2^{40}\log 2\ <\ 4.94\times 10^{15}.
\]
Everything downstream is an elementary expression in this number: the
coefficient $\kappa_{p}$ and the constant $C_{p}$ of Theorem~\ref{thm:aperiodic}
(Remark~\ref{rem:constA}) and the threshold $m_{0}(n)$ of
Theorem~\ref{thm:complexity} (Remark~\ref{rem:constB}). We record the resulting
numerical values as they arise; no estimate below is merely qualitative, and no
step of either proof appeals to a compactness, limit or contradiction argument
that would destroy effectivity. The formalization uses the same definition:
\texttt{BakerWustholz.C n d} is the displayed formula, and the Lean proofs of
both theorems instantiate it at $(4,1)$.

The constants are, of course, far from optimal: $\bwC(4,1)$ is a general-purpose
transcendence constant, and no attempt is made here to sharpen it. What matters
for the statements is that they are computable, so that the qualifiers ``for a
constant $C_{p}$'' and ``for all sufficiently large $m$'' can be replaced by
numbers.
\end{remark}

The estimates behind Theorems~\ref{thm:stewart} and~\ref{thm:blocks} bound how
far below the top of the expansion a run of \emph{constant} bits can begin: a run
of zeros (resp.\ ones) in positions $[x,y)$ forces $(\Mtop-x)\log 2\le\bwC(3,1)
\log 3\,(\Mtop+1-y)\max(\log 2m,1)$, via a three-term linear form.\footnote{We use
the descriptive names \emph{run estimate} for these lemmas and
\emph{period-$p$ run estimate} for Lemma~\ref{lem:gapP}; the formalization calls
them ``gap principles''. Each is an instance of Theorem~\ref{thm:bw}.} The result
of this paper needs the analogous estimate for \emph{periodic} runs, which is new
and uses a four-term form.

The combinatorial half of that estimate is the following remainder identity,
which converts a periodic window into an exact equation between integers. It is
where the auxiliary integer $2^{p}-1$---the fourth logarithm---originates.

\begin{lemma}[Periodic remainder identity]\label{lem:permod}
Let $N,x,y,p\in\N$ with $p\ge 1$, and suppose $\bit{i}(N)=\bit{i+p}(N)$ for all
$i$ with $x\le i$ and $i+p<y$. Put $c:=\lfloor N/2^{x}\rfloor\bmod 2^{p}$. Then
for every $t\ge 0$ with $x+tp\le y$,
\[
(2^{p}-1)\bigl(N\bmod 2^{x+tp}\bigr)
\ =\ (2^{p}-1)\bigl(N\bmod 2^{x}\bigr)+c\,\bigl(2^{x+tp}-2^{x}\bigr).
\]
\end{lemma}

\begin{proof}
We first record that the $p$-bit block read off at each level of the window is
the same, namely $c$: for every $j$ with $x+jp+p\le y$,
\begin{equation}\label{eq:block}
\lfloor N/2^{x+jp}\rfloor\bmod 2^{p}\ =\ c .
\end{equation}
Indeed, the block at level $x+jp$ consists of the bits
    \[
\bit{x+jp}(N),\dots,\bit{x+jp+p-1}(N),
\]
    and each of these equals its counterpart
$p$ places higher---all the indices involved satisfy the hypothesis, because
$x+jp+p\le y$. So the block at level $x+jp$ equals the block at level $x+jp+p$,
and induction on $j$ from the base case $j=0$, where the block is $c$ by
definition, gives \eqref{eq:block}.

Next, peeling one $p$-bit block off a remainder is the identity
    \[
    N\bmod 2^{a+p}=\bigl(N\bmod 2^{a}\bigr)+\bigl(\lfloor N/2^{a}\rfloor\bmod
2^{p}\bigr)2^{a},
\]
    valid for every $a$. Taking $a=x+tp$ and inserting
\eqref{eq:block},
\begin{equation}\label{eq:peel}
N\bmod 2^{x+(t+1)p}\ =\ \bigl(N\bmod 2^{x+tp}\bigr)+c\,2^{x+tp} .
\end{equation}
Now induct on $t$. For $t=0$ both sides of the assertion equal
$(2^{p}-1)(N\bmod 2^{x})$. Assuming it for $t$, multiply \eqref{eq:peel} by
$2^{p}-1$; the new term is 
    \[
    (2^{p}-1)c\,2^{x+tp}=c\bigl(2^{x+(t+1)p}-2^{x+tp}\bigr),
\]
and it combines with the inductive hypothesis by the telescoping
    \[
        \bigl(2^{x+tp}-2^{x}\bigr)+\bigl(2^{x+(t+1)p}-2^{x+tp}\bigr)=2^{x+(t+1)p}-2^{x}.\]
\end{proof}

\begin{remark}
Unrolled, Lemma~\ref{lem:permod} says
    \[
    N\bmod 2^{x+tp}=\bigl(N\bmod 2^{x}\bigr)+c\,2^{x}\,(2^{tp}-1)/(2^{p}-1):
\]
    the
window contributes $t$ copies of the pattern $c$, a geometric series. The point
of stating it in the cleared form is that multiplying by $2^{p}-1$ removes the
denominator, so every quantity in sight stays a rational integer; the price is
the extra factor $2^{p}-1$, which is exactly the fourth logarithm in
Lemma~\ref{lem:gapP} below and the source of its $\bwC(4,1)$.
\end{remark}

The arithmetic half is the following nondegeneracy statement, which is what
keeps the linear form of Lemma~\ref{lem:gapP} away from zero. We state it in
base $b$, both because nothing is gained by specializing and because the general
form is what Section~\ref{sec:scope} needs.

\begin{lemma}[Nondegeneracy of the remainder]\label{lem:nondeg}
Let $b\ge 2$, $p\ge 1$ and $x\ge 0$ be integers, let $N\ge 1$, and put
\[
\begin{gathered}
G:=b^{p}-1,\qquad A_{2}:=N\bmod b^{x},\qquad
c:=\lfloor N/b^{x}\rfloor\bmod b^{p},\\[2pt]
E:=G\,A_{2}-c\,b^{x}.
\end{gathered}
\]
If $\ell$ is a prime with $\ell\mid b$ and $\ell\nmid N$, then
\[
E\ \equiv\ -N\ \pmod{\ell};
\]
in particular $E\neq 0$.
\end{lemma}

\begin{proof}
The two cases give the same congruence. If $x\ge 1$ then $\ell\mid b^{x}$, so
$c\,b^{x}\equiv 0$ and $A_{2}=N\bmod b^{x}\equiv N$, while $G=b^{p}-1\equiv-1$;
hence $E\equiv-N$. If $x=0$ then $A_{2}=N\bmod 1=0$ and $b^{x}=1$, so
$E=-c=-(N\bmod b^{p})$, and $\ell\mid b^{p}$ gives $N\bmod b^{p}\equiv N$; hence
again $E\equiv-N$. As $\ell\nmid N$, we get $E\not\equiv 0\pmod{\ell}$, so
$E\neq 0$.
\end{proof}

\begin{remark}[Degeneracy, and the role of $x\ge 1$]\label{rem:nondeg}
Two comments on the hypotheses, the second of which we have not pursued.

\emph{(i) What the hypothesis excludes is the trailing-digit degeneracy.} If
every prime of $b$ divides $N$, then $E$ genuinely can vanish. For $b=2$ and
$N=6^{m}=2^{m}3^{m}$, any $x$ and $p$ with $x+p\le m$ give $A_{2}=0$ and $c=0$,
hence $E=0$: at $m=10$, $x=3$, $p=4$ one has $E=0$ on the nose, and the identity
\eqref{eq:key} below degenerates to $0=0$. Nothing is lost, because the deep
windows of such an $N$ really are periodic---they are constant---and no linear
form could say otherwise. This is the mechanism behind the $(a,b)=(6,2)$
discussion of Section~\ref{sec:scope}.

\emph{(ii) The condition $x\ge 1$ of Lemma~\ref{lem:gapP} is not needed for
nondegeneracy.} Lemma~\ref{lem:nondeg} covers $x=0$, where the two contributions
to $E$ simply exchange roles: $A_{2}$ vanishes and all of $E$ is carried by
$-c=-(N\bmod b^{p})$. Since $x\ge 1$ enters the proof of Lemma~\ref{lem:gapP}
nowhere else---see Remark~\ref{rem:hyp}---the lemma should hold with $1\le x$
weakened to $0\le x$, and Theorem~\ref{thm:aperiodic} would then carry the
offset $-1$ in place of $-2$, no window being sacrificed, matching
Theorems~\ref{thm:stewart} and~\ref{thm:blocks} exactly. We have not made that
change. The formalized \texttt{gap\_principle} retains the hypothesis
$1\le x$---it is used there at exactly two points, both inside the parity
computation---and the verified form of Theorem~\ref{thm:aperiodic} is the one
stated above, with $-2$. The improvement is recorded here as an observation, not
claimed as a theorem of this paper.
\end{remark}

\begin{lemma}[Period-$p$ run estimate]\label{lem:gapP}
Let $m\ge 1$, $p\ge 1$, $1\le x<y\le\Mtop$ with $2p\le y-x$, and suppose
$\bit{i}(3^{m})=\bit{i+p}(3^{m})$ for all $i$ with $x\le i$ and $i+p<y$. Then
\[
(\Mtop-x)\log 2\ \le\ \bwC(4,1)\,p\,\log 3\,(\Mtop+2p-y)\,
\max\bigl(\log(2m+p),1\bigr)+(p+1)\log 2 .
\]
\end{lemma}

\begin{proof}
Write $N:=3^{m}$, so that $2^{\Mtop}\le N<2^{\Mtop+1}$, and put $G:=2^{p}-1$.

\emph{Step 1: trimming the window to whole periods.} Let
$t:=\lfloor(y-x)/p\rfloor$ and $y':=x+tp$. Then $y'\le y$ and
\begin{equation}\label{eq:trim}
y-y'\ =\ (y-x)\bmod p\ \le\ p-1 .
\end{equation}
The hypothesis $2p\le y-x$ gives $t\ge 2$, whence $x<x+2p\le y'\le y\le\Mtop$;
in particular
\begin{equation}\label{eq:xM}
x+2\ \le\ x+2p\ \le\ \Mtop .
\end{equation}
The trimmed window $[x,y')$ is still $p$-periodic and now consists of exactly
$t\ge 2$ full periods.

\emph{Step 2: the repeated pattern.} Put $A_{2}:=N\bmod 2^{x}$, the digits
strictly below the window, and $c:=\lfloor N/2^{x}\rfloor\bmod 2^{p}$, the
pattern that repeats. Lemma~\ref{lem:permod}, applied with this $t$, gives
\[
G\,\bigl(N\bmod 2^{y'}\bigr)\ =\ G\,A_{2}+c\,\bigl(2^{y'}-2^{x}\bigr).
\]

\emph{Step 3: an exact integer identity.} Put $A_{1}'':=\lfloor N/2^{y'}\rfloor$,
the digits at or above $y'$, and $A_{1}':=G\,A_{1}''+c$. Division with remainder,
$N=A_{1}''2^{y'}+\bigl(N\bmod 2^{y'}\bigr)$, combined with Step 2, gives
$GN=G A_{1}''2^{y'}+G A_{2}+c\,2^{y'}-c\,2^{x}$, that is,
\begin{equation}\label{eq:key}
(2^{p}-1)\,3^{m}\ =\ A_{1}'\,2^{y'}+E,\qquad E:=G\,A_{2}-c\,2^{x}.
\end{equation}

\emph{Step 4: parity nondegeneracy, $E\neq 0$.} This is
Lemma~\ref{lem:nondeg} with $b=2$ and $\ell=2$, legitimate because $3^{m}$ is
odd: it gives $E\equiv-3^{m}\equiv 1\pmod 2$, so $E$ is odd and in particular
nonzero. Concretely, $G=2^{p}-1$ is odd and $A_{2}=3^{m}\bmod 2^{x}$ is odd
because $x\ge 1$, so $G A_{2}$ is odd, while $c\,2^{x}$ is even for the same
reason. It is this parity that replaces, in the periodic setting, the appeal to
the irrationality of $\log 2/\log 3$ making the corresponding form nonzero for
constant runs. The hypothesis $x\ge 1$ is used nowhere else in this proof, and
Lemma~\ref{lem:nondeg} does not require it; see Remark~\ref{rem:nondeg}(ii).

\emph{Step 5: the size of $A_{1}'$.} From $2^{\Mtop}\le N$ we get
$A_{1}''\ge 2^{\Mtop-y'}$, hence $A_{1}'\ge G\,2^{\Mtop-y'}\ge 1$. From
$N<2^{\Mtop+1}$ we get $A_{1}''<2^{\Mtop+1-y'}$, so, using $G<2^{p}$ and
$c<2^{p}$,
\begin{equation}\label{eq:A1ub}
A_{1}'\ <\ 2^{p}A_{1}''+2^{p}\ =\ 2^{p}\bigl(A_{1}''+1\bigr)
\ \le\ 2^{\,\Mtop+1+p-y'} .
\end{equation}

\emph{Step 6: the ratio lies within $\tfrac12$ of $1$.} Set
$R:=G\,3^{m}/\bigl(A_{1}'2^{y'}\bigr)>0$, so that $R-1=E/\bigl(A_{1}'2^{y'}\bigr)$
by \eqref{eq:key}. For the numerator, $0\le G A_{2}<2^{p}2^{x}$ and
$0\le c\,2^{x}<2^{p}2^{x}$, so $|E|\le 2^{x+p}$. For the denominator, Step 5
gives $A_{1}'2^{y'}\ge G\,2^{\Mtop-y'}2^{y'}=G\,2^{\Mtop}\ge 2^{p-1}2^{\Mtop}$.
Hence
\begin{equation}\label{eq:Rm1}
|R-1|\ \le\ \frac{2^{x+p}}{2^{p-1}\,2^{\Mtop}}\ =\ 2^{\,x+1-\Mtop}
\ \le\ \tfrac12 ,
\end{equation}
the last inequality by \eqref{eq:xM}.

\emph{Step 7: the four-term form and its upper bound.} By Step 4 we have
$E\neq 0$, so $R\neq 1$, and therefore
\[
\Lambda\ :=\ \log(2^{p}-1)+m\log 3-y'\log 2-\log A_{1}'\ =\ \log R
\]
is nonzero. Since $R\ge\tfrac12$ by \eqref{eq:Rm1}, the two-sided estimate
$|\log R|\le 2|R-1|$ applies: for $R\ge 1$ it is $\log R\le R-1$, and for $R<1$
it follows from $-\log R=\log(1/R)\le 1/R-1=(1-R)/R\le 2(1-R)$. Combining with
\eqref{eq:Rm1},
\begin{equation}\label{eq:upper}
0<|\Lambda|\le 2^{\,x+2-\Mtop},\qquad\text{hence}\qquad
\log|\Lambda|\ \le\ (x+2-\Mtop)\log 2 .
\end{equation}

\emph{Step 8: the lower bound from Theorem~\ref{thm:bw}.} Apply
Theorem~\ref{thm:bw} over $K=\Q$, $d=1$, with $n=4$ and
\[
(\alpha_{1},\dots,\alpha_{4})=\bigl(2^{p}-1,\,3,\,2,\,A_{1}'\bigr),\qquad
(b_{1},\dots,b_{4})=\bigl(1,\,m,\,-y',\,-1\bigr),
\]
\[
B:=2m+p\ \ge\ 2 .
\]
All four $\alpha_{i}$ are nonzero positive rationals by Step 5, and the
coefficients are admissible: the only one needing comment is
$|b_{3}|=y'\le\Mtop\le 2m$, where $\Mtop\le 2m$ because $3^{m}\le 4^{m}=2^{2m}$.
For the heights, recall $\mh(a)=\max(\log a,1)$ for a positive integer $a$, and
$\log 2\le 1$:
\[
\mh(2^{p}-1)\le\max(p\log 2,1)\le p,\qquad \mh(3)=\log 3,\qquad \mh(2)=1,
\]
the first using $p\ge 1$; and by \eqref{eq:A1ub},
\[
\mh(A_{1}')\ \le\ \max\bigl(\log A_{1}',1\bigr)\ \le\ \Mtop+1+p-y',
\]
where the floor $1$ is harmless because $y'\le\Mtop$ and $p\ge 1$ force
$\Mtop+1+p-y'\ge 1$. Hence
$\prod_{i}\mh(\alpha_{i})\le p\log 3\,(\Mtop+1+p-y')$, and since $\Lambda\neq 0$,
\begin{equation}\label{eq:lower}
\log|\Lambda|\ \ge\ -\bwC(4,1)\,p\,\log 3\,(\Mtop+1+p-y')\,
\max\bigl(\log(2m+p),1\bigr).
\end{equation}

\emph{Step 9: conclusion.} Chaining \eqref{eq:upper} and \eqref{eq:lower} and
using the identity $(\Mtop-x)\log 2=2\log 2-(x+2-\Mtop)\log 2$,
\[
(\Mtop-x)\log 2\ \le\ \bwC(4,1)\,p\,\log 3\,(\Mtop+1+p-y')\,
\max\bigl(\log(2m+p),1\bigr)+2\log 2 .
\]
Finally \eqref{eq:trim} gives $\Mtop+1+p-y'\le\Mtop+2p-y$, and
$2\log 2\le(p+1)\log 2$ because $p\ge 1$. This is the asserted inequality.
\end{proof}

\begin{remark}[Where the hypotheses are used]\label{rem:hyp}
Each hypothesis of Lemma~\ref{lem:gapP} enters at exactly one place, and it is
worth isolating them.
\begin{itemize}
\item $x\ge 1$ is used only in Step 4, and only to keep the parity computation
in the form given there; Lemma~\ref{lem:nondeg} delivers $E\neq 0$ without it.
As stated, this single requirement is what forces the sacrifice of one depth
window in the proof of Theorem~\ref{thm:aperiodic}, hence the offset $-2$ there
in place of $-1$---a loss that Remark~\ref{rem:nondeg}(ii) argues is avoidable.
\item $2p\le y-x$ is used twice, both in Step 1: to obtain $t\ge 2$, so that the
trimmed window is nonempty, and to obtain \eqref{eq:xM}, which is precisely what
places $R$ within $\tfrac12$ of $1$ in \eqref{eq:Rm1} and so licenses the
two-sided logarithm estimate of Step 7. A window shorter than two periods
carries no information here.
\item $y\le\Mtop$ is used in Step 5, both for the lower bound
$A_{1}''\ge 2^{\Mtop-y'}$ and for $\Mtop+1+p-y'\ge 1$ in Step 8.
\item The trimming of Step 1 costs at most an additive $p-1$ in the depth
factor; this is absorbed in Step 9 by replacing $\Mtop+1+p-y'$ with the larger
$\Mtop+2p-y$, which is the only reason the stated bound carries $2p$ rather
than $p$.
\end{itemize}
At $p=1$ we have $G=1$, so the leading term $\log(2^{p}-1)$ vanishes and
$\Lambda$ degenerates to the three-term form behind
Theorems~\ref{thm:stewart} and~\ref{thm:blocks}. The lemma nonetheless invokes
Theorem~\ref{thm:bw} with $n=4$, so it does not recover the classical constant
$\bwC(3,1)$ at $p=1$; this is the one respect in which Theorem~\ref{thm:aperiodic}
is weaker, at $p=1$, than the Theorem~\ref{thm:blocks} it formally contains.
\end{remark}

Stewart's counting device turns such estimates into an effective lower bound.
Fix $\theta\ge 2$ and consider, for $j=0,1,2,\dots$, the depth windows of bit
positions
\[
\bigl[\,\Mtop-\theta^{\,j+1},\ \Mtop-\theta^{\,j}\,\bigr);
\]
there are $\lfloor\log_{\theta}\Mtop\rfloor$ disjoint such windows below the top
bit. If a counting function is bounded below by the number of windows carrying a
witness, an effective lower bound follows with a single constant.

\begin{proposition}\label{prop:endgame}
Let $\kappa>0$ and $Q\colon\N\to\N$. Suppose that for all $m\ge 2$ and all
$\theta\ge 2$ with $\kappa\log m<\theta\log 2$ one has
$\lfloor\log_{\theta}\lfloor\log_{2}3^{m}\rfloor\rfloor\le Q(m)$. Then there is a
constant $C>0$, which one may take to be $\max\bigl(\log\frac{\kappa+2}{\log 2},
1\bigr)$, such that
\[
Q(m)\ \ge\ \frac{\log m}{\log\log m+C}-1
\qquad\text{for every } m\ge 2 .
\]
\end{proposition}

\begin{proof}[Proof sketch]
Choose $\theta\approx C'\log m$ just above the threshold, so $\theta\le C'\log m$
for a constant $C'$; then $k:=\lfloor\log_{\theta}\Mtop\rfloor\le Q(m)$ by
hypothesis. Since $3^{m}\ge 2^{m}$ gives $\Mtop\ge m$, one has $\log m\le
(k+1)\log\theta$ and $\log\theta\le\log\log m+C$; dividing yields the bound.
Positivity of the denominator uses $\log\log 2>-1$.
\end{proof}

\begin{lemma}\label{lem:tendsto}
If $Q\colon\N\to\N$ satisfies $Q(m)\ge\log m/(\log\log m+C)-1$ for all $m\ge 2$
and some $C>0$, then $Q(m)\to\infty$.
\end{lemma}

\begin{proof}[Proof sketch]
The right-hand side tends to infinity, since $\log m$ dominates $\log\log m$; the
quantitative threshold uses $\log\log m\le 2\sqrt{\log m}$.
\end{proof}

%======================================================================
\section{Aperiodicity: period-\texorpdfstring{$p$}{p} breaks}\label{sec:aperiodic}
%======================================================================

We prove Theorem~\ref{thm:aperiodic}: for each fixed period $p$, the number of
places where the binary expansion of $3^{m}$ fails period $p$ grows at the
Stewart rate.

The proof is a fixed procedure with two interchangeable parts: the estimate
(Lemma~\ref{lem:gapP}), which certifies that one prescribed window carries a
break, and the counting scheme (Proposition~\ref{prop:endgame}), which converts
``one break per window'' into a growth rate. It is worth setting the procedure
out in full, since the passage from windows to breaks is where the two meet.

\medskip
\noindent\textbf{The scheme.} Fix $p\ge 1$ and $m\ge 2$. Write
$\Mtop=\Mtop(3^{m})$, and let $\theta\ge 2$ be any window size satisfying the
\emph{admissibility condition}
\begin{equation}\label{eq:adm}
\begin{gathered}
\kappa_{p}\,\log m\ <\ \theta\,\log 2,\qquad\text{where}\\[4pt]
\kappa_{p}\ :=\ \underbrace{\bwC(4,1)\log 3\cdot p(1+2p)\,\beta_{p}+(p+1)}
_{\text{needed in (W3)}}\ +\ \underbrace{2p}_{\text{needed in (W2)}} ,
\end{gathered}
\end{equation}
where $\beta_{p}=2+(\log(2+p)+1)/\log 2$; note $\kappa_{p}\asymp p^{2}\log p$.
Put $k:=\lfloor\log_{\theta}\Mtop\rfloor$ and, for $0\le j<k-1$,
\[
W_{j}\ :=\ \bigl[\,\Mtop-\theta^{\,j+1},\ \Mtop-\theta^{\,j}\,\bigr),
\]
so that the $W_{j}$ tile the expansion from the top downwards in geometrically
growing steps:
\[
\underbrace{\bigl[\,0,\ \Mtop-\theta^{\,k-1}\bigr)}_{\text{unused}}\ \
\cdots\ \
\underbrace{\bigl[\,\Mtop-\theta^{2},\,\Mtop-\theta\bigr)}_{W_{1}}\ \
\underbrace{\bigl[\,\Mtop-\theta,\,\Mtop\bigr)}_{W_{0}} .
\]

\begin{itemize}
\item[\textbf{(W1)}] \emph{Geometry.} The $W_{j}$ are pairwise disjoint
subintervals of $[0,\Mtop)$; $W_{j}$ has width $\theta^{j+1}-\theta^{j}
=\theta^{j}(\theta-1)$ and begins at depth $\theta^{j+1}$ below the top bit.
There are $k-1$ of them, not $k$: the index $j=k-1$ is discarded because
Lemma~\ref{lem:gapP} requires base position $\ge 1$, i.e.\ $\theta^{j+1}<\Mtop$
strictly. This is the sacrificed window, and the whole source of the offset
$-2$ (see Remark~\ref{rem:nondeg}(ii)).

\item[\textbf{(W2)}] \emph{Admissibility gives width.} Since $\kappa_{p}\ge 2p$
and $\log m\ge\log 2$, condition \eqref{eq:adm} forces $\theta>2p$. Hence every
$W_{j}$ has width $\theta^{j}(\theta-1)\ge\theta-1\ge 2p$, which is exactly the
hypothesis $2p\le y-x$ of Lemma~\ref{lem:gapP}. This is what the summand $2p$
in $\kappa_{p}$ is for.

\item[\textbf{(W3)}] \emph{Each window carries a break.} Suppose some $W_{j}$
were $p$-periodic. Apply Lemma~\ref{lem:gapP} with $x=\Mtop-\theta^{j+1}$ and
$y=\Mtop-\theta^{j}$, so that $\Mtop-x=\theta^{j+1}$ and
$\Mtop+2p-y=\theta^{j}+2p\le(1+2p)\theta^{j}$:
\[
\theta^{j+1}\log 2\ \le\ \bwC(4,1)\,p\log 3\,(1+2p)\,\theta^{j}\,
\max\bigl(\log(2m+p),1\bigr)+(p+1)\log 2 .
\]
Divide by $\theta^{j}\ge 1$ and use $\max(\log(2m+p),1)\le\beta_{p}\log m$
together with $\log 2\le\log m$:
\[
\theta\log 2\ \le\ \bigl[\bwC(4,1)\log 3\cdot p(1+2p)\beta_{p}+(p+1)\bigr]\log m
\ \le\ \kappa_{p}\log m ,
\]
contradicting \eqref{eq:adm}. So $W_{j}$ is not $p$-periodic: there is an index
$i_{j}$ with $\Mtop-\theta^{j+1}\le i_{j}$, $i_{j}+p<\Mtop-\theta^{j}$ and
$\bit{i_{j}}(3^{m})\neq\bit{i_{j}+p}(3^{m})$. This is the one step that
consumes arithmetic; everything else is bookkeeping.

\item[\textbf{(W4)}] \emph{The injection.} The assignment $j\mapsto i_{j}$ is
strictly decreasing, hence injective: if $a<b$ then
\[
i_{b}\ <\ i_{b}+p\ <\ \Mtop-\theta^{\,b}\ \le\ \Mtop-\theta^{\,a+1}\ \le\ i_{a} .
\]
Each $i_{j}$ lies in $[0,\Mtop)$ and is by construction a period-$p$ break, so
$j\mapsto i_{j}$ injects $\{0,\dots,k-2\}$ into the set counted by
$\bcount_{p}(3^{m})$. Comparing cardinalities,
\begin{equation}\label{eq:inject}
\Bigl\lfloor\log_{\theta}\Mtop\Bigr\rfloor-1\ =\ k-1\ \le\ \bcount_{p}(3^{m}),
\qquad\text{i.e.}\qquad
\Bigl\lfloor\log_{\theta}\Mtop\Bigr\rfloor\ \le\ \bcount_{p}(3^{m})+1 .
\end{equation}

\item[\textbf{(W5)}] \emph{The endgame.} Steps (W1)--(W4) were carried out for
an arbitrary $\theta$ satisfying \eqref{eq:adm}, so \eqref{eq:inject} holds for
\emph{every} admissible $\theta$. That is precisely the hypothesis of
Proposition~\ref{prop:endgame}, applied to $Q(m)=\bcount_{p}(3^{m})+1$ with
coefficient $\kappa_{p}$.
\end{itemize}

\begin{proof}[Proof of Theorem~\ref{thm:aperiodic}]
Run the scheme. By (W5) and Proposition~\ref{prop:endgame},
\[
\bcount_{p}(3^{m})+1\ \ge\ \frac{\log m}{\log\log m+C_{p}}-1
\qquad\text{for all } m\ge 2,
\]
with $C_{p}=\max\bigl(\log((\kappa_{p}+2)/\log 2),1\bigr)>0$; subtracting $1$
gives the stated inequality. The divergence $\bcount_{p}(3^{m})\to\infty$ is
Lemma~\ref{lem:tendsto}.
\end{proof}

\begin{remark}[The division of labour]\label{rem:division}
The scheme isolates the two inputs cleanly. Lemma~\ref{lem:gapP} is used exactly
once, in (W3), and only through the single implication ``a window at depth
$\theta^{j+1}$ and of width $\ge 2p$ cannot be $p$-periodic''; it never sees the
counting. Proposition~\ref{prop:endgame} is used exactly once, in (W5), and only
through \eqref{eq:inject}; it never sees the digits---in the formalization all
Baker content is quarantined in its hypothesis, which is why
\texttt{windowCount\_lower\_bound\_gen} carries the footprint std3 while
\texttt{gap\_principle} carries std3~+~[BW93]. The coefficient $\kappa_{p}$ is
the only channel between them, and \eqref{eq:adm} displays it doing its two
jobs: the bracket is what (W3) must beat, the summand $2p$ is what (W2) needs.
Replacing Lemma~\ref{lem:gapP} by a sharper estimate would change only
$\kappa_{p}$, leaving (W1), (W2), (W4) and (W5) untouched---which is the sense
in which Remark~\ref{rem:ballpark}'s discussion of better transcendence input is
a statement about $\kappa_{p}$ alone.
\end{remark}

\begin{remark}
The offset is $-2$, rather than the $-1$ of Theorems~\ref{thm:stewart}
and~\ref{thm:blocks}, because one window is sacrificed: Lemma~\ref{lem:gapP}
requires the base position of the periodic stretch to be at least $1$. The bound
is uniform in $m$ for each fixed $p$; the growth $C_{p}\asymp\log p$ of the
constant is carried by the coefficient $\kappa_{p}\asymp p^{2}\log p$.
\end{remark}

\begin{remark}[The constant $C_{p}$, numerically]\label{rem:constA}
The proof exhibits $C_{p}$, via Proposition~\ref{prop:endgame}, as
\[
C_{p}=\max\Bigl(\log\frac{\kappa_{p}+2}{\log 2},\,1\Bigr),\qquad
\kappa_{p}=\bwC(4,1)\log 3\cdot p(1+2p)\,\beta_{p}+3p+1,
\]
    \[
\beta_{p}=2+\frac{\log(2+p)+1}{\log 2},
\]
so Remark~\ref{rem:explicit} turns it into a number. For $p=1$,
$\beta_{1}<5.028$, $\kappa_{1}<8.19\times 10^{16}$ and
\[
C_{1}<39.31,
\]
i.e.\ $\bcount_{1}(3^{m})=\tcount(3^{m})\ge\log m/(\log\log m+39.31)-2$ for all
$m\ge 2$. In general $\kappa_{p}<5.43\times10^{15}\,p(2p+1)\beta_{p}$ and
\[
C_{p}\ <\ 40+3\log(p+1)\qquad\text{for every }p\ge 1 .
\]
The size of these constants is dominated entirely by $\bwC(4,1)$; the
combinatorial part of the argument contributes the factor $p(2p+1)\beta_{p}$,
of order $p^{2}\log p$.
\end{remark}

\noindent
The first ten periods, evaluated from \eqref{eq:adm} and
Proposition~\ref{prop:endgame}:

\begin{center}
\footnotesize
\renewcommand{\arraystretch}{1.15}
\setlength{\tabcolsep}{7pt}
\begin{tabular}{@{}rrrrr@{}}
\hline
$p$ & $\beta_{p}$ & $\kappa_{p}$ & $C_{p}$
  & least $m$ with $\bcount_{p}(3^{m})\ge 1$\\
\hline
$1$  & $5.028$ & $8.18\times10^{16}$ & $39.31$ & $1.1\times10^{38}$\\
$2$  & $5.443$ & $2.95\times10^{17}$ & $40.59$ & $1.5\times10^{39}$\\
$3$  & $5.765$ & $6.57\times10^{17}$ & $41.39$ & $7.6\times10^{39}$\\
$4$  & $6.028$ & $1.18\times10^{18}$ & $41.98$ & $2.5\times10^{40}$\\
$5$  & $6.250$ & $1.87\times10^{18}$ & $42.44$ & $6.4\times10^{40}$\\
$6$  & $6.443$ & $2.73\times10^{18}$ & $42.82$ & $1.4\times10^{41}$\\
$7$  & $6.613$ & $3.77\times10^{18}$ & $43.14$ & $2.7\times10^{41}$\\
$8$  & $6.765$ & $4.99\times10^{18}$ & $43.42$ & $4.8\times10^{41}$\\
$9$  & $6.902$ & $6.40\times10^{18}$ & $43.67$ & $7.9\times10^{41}$\\
$10$ & $7.028$ & $8.01\times10^{18}$ & $43.89$ & $1.3\times10^{42}$\\
\hline
\end{tabular}
\end{center}

\noindent
The last column is the least $m$ at which
$\log m/(\log\log m+C_{p})-2$ becomes positive, i.e.\ the point from which
Theorem~\ref{thm:aperiodic} asserts anything at all. Two features are worth
noting. First, $C_{p}$ barely moves: it is a logarithm of $\kappa_{p}$, so the
factor $p^{2}\log p$ by which $\kappa_{p}$ grows across the table costs $C_{p}$
less than five units. Second, that slow growth is nevertheless amplified on
exponentiation, and the last column climbs some four orders of magnitude over
$p\le 10$. Both columns are ordinary evaluations of the displayed formulas;
Remark~\ref{rem:ballpark} discusses what such magnitudes mean.

%======================================================================
\section{Subword complexity}\label{sec:complexity}
%======================================================================

We now prove Theorem~\ref{thm:complexity}. The combinatorial input is the
following finite form of the Morse--Hedlund theorem, in which we make no
appeal to Fine--Wilf; it is proved by a self-contained determinism-propagation
argument.

\begin{lemma}[Finite Morse--Hedlund floor]\label{lem:fmh}
Let $u=(u_{0},u_{1},\dots)$ be a word over an alphabet with decidable equality,
and let $n\ge 1$ and $L\ge 3n$. If $u$ has at most $n$ distinct factors of length
$n$ among the positions $0,\dots,L-n$, then there exist $a$ and $1\le p\le n$ such
that $u$ has period $p$ on the factor $[a,\,a+(L-2n))$; that is,
$u_{t}=u_{t+p}$ for all $a\le t$ with $t+p<a+(L-2n)$.
\end{lemma}

\begin{proof}[Proof sketch]
On the fixed position set $[0,L-n]$ the number $c(k)$ of distinct length-$k$
factors is non-decreasing in $k$: dropping the last letter of a length-$(k+1)$
factor surjects onto the length-$k$ factors. Since $c(0)=1$ and $c(n)\le n$, the
count cannot strictly increase $n$ times, so there is a plateau $c(k)=c(k+1)$ with
$k<n$. A plateau forces right-determinism---two positions carrying the same
length-$k$ factor carry the same next letter. Pigeonhole over the first $c(k)+1$
positions produces two equal length-$k$ factors at an offset $p\le n$, and
determinism propagates the period-$p$ relation across the window.
\end{proof}

Feeding such a periodic factor to the period-$p$ run estimate caps its depth,
in the following effective form.

\begin{lemma}\label{lem:core}
Let $n\ge 1$ and $\Mtop=\lfloor\log_{2}3^{m}\rfloor\ge 4n+1$. If
$\pcx_{3^{m}}(n)\le n$, then
\[
\bigl(\Mtop-(3n+2)\bigr)\log 2\ \le\ 4\,\bwC(4,1)\,n^{2}\,\log 3\,
\max\bigl(\log(2m+n),1\bigr).
\]
\end{lemma}

\begin{proof}[Proof sketch]
By Lemma~\ref{lem:fmh} with $L=\Mtop$, a factor complexity $\le n$ yields a
period-$p$ factor with $1\le p\le n$ of length $\Mtop-2n$, starting at some
$a\le 2n$. Apply the period-$p$ run estimate (Lemma~\ref{lem:gapP}) with base
$x=a+1\ge 1$ and $y=a+(\Mtop-2n)$: the depth factor is
$\Mtop+2p-y=2p+2n-a\le 4n$, and $p(2p+2n-a)\le 4n^{2}$ collapses the right-hand
side to the stated form.
\end{proof}

\begin{proof}[Proof of Theorem~\ref{thm:complexity}]
Fix $n\ge 1$ and suppose, for contradiction, that $\pcx_{3^{m}}(n)\le n$ for
arbitrarily large $m$. Since $\Mtop=\lfloor\log_{2}3^{m}\rfloor\ge m$,
Lemma~\ref{lem:core} would give
\[
(m-(3n+2))\log 2\ \le\ 4\,\bwC(4,1)\,n^{2}\log 3\,\max(\log(2m+n),1);
\]
but the left-hand side grows linearly in $m$ while the right-hand side grows like
$\log m$, a contradiction for large $m$. Hence $\pcx_{3^{m}}(n)\ge n+1$
eventually. The effective threshold $m_{0}(n)$ is where $m$ overtakes
\[
4\,\bwC(4,1)\,n^{2}\,\frac{\log 3}{\log 2}\,\log(2m+n)+(3n+2).\qedhere
\]
\end{proof}

\begin{remark}[What Theorem~\ref{thm:complexity} does and does not
say]\label{rem:calibration}
The word in question has length $\Mtop=\Mtop(3^{m})$, so its complexity function
is confined to
\[
1\ \le\ \pcx_{3^{m}}(n)\ \le\ \min\bigl(2^{n},\,\Mtop-n+1\bigr),
\]
the upper bound being trivial: there are only $2^{n}$ binary words of length $n$,
and only $\Mtop-n+1$ positions at which to read one. Within that range the value
$n+1$ sits at the very bottom of the aperiodic regime, not near the top.
\begin{itemize}
\item By Morse--Hedlund \cite{MH38}---in the finite form of
Lemma~\ref{lem:fmh}---the inequality $\pcx_{w}(n)\le n$ for a \emph{single} $n$
already forces periodicity. So ``$\pcx(n)\ge n+1$ for every $n$'' is exactly the
negation of periodicity, and asserts nothing beyond it.
\item The bound is attained identically: Sturmian words satisfy
$\pcx_{w}(n)=n+1$ for all $n$, and they are precisely the aperiodic words of
\emph{minimal} complexity. A word meeting our bound with equality is thus as far
from maximal complexity as an aperiodic word can be.
\item Maximal complexity, $\pcx(n)=\min(2^{n},\Mtop-n+1)$, is a de Bruijn-type
property; a normal or random word has $\pcx(n)=2^{n}$ throughout the range
$2^{n}\lesssim\Mtop$. We prove nothing of the sort, and nothing intermediate
either.
\end{itemize}
Theorem~\ref{thm:complexity} should therefore be read as the quantitative,
$n$-uniform statement that \emph{the low-order digit word of $3^{m}$ is not
periodic}, valid for all $m$ beyond an explicit threshold. Whether
$\pcx_{3^{m}}(n)/n\to\infty$---let alone $\pcx_{3^{m}}(n)=2^{n}$ in the
admissible range---was open when this was written; the first of the two has
since been settled, ineffectively, in \cite{Bug26}. See
Section~\ref{sub:added}.

To be clear, the combinatorial half of the argument imposes this ceiling:
Lemma~\ref{lem:fmh}
converts the hypothesis $\pcx(n)\le n$ into a long periodic factor, whereas a
hypothesis of the shape $\pcx(n)\le Cn$ with $C>1$ yields no periodic factor at
all, and so gives Lemma~\ref{lem:gapP} nothing to act on. Sharpening the
linear-forms input would not move the bound past $n+1$. A different arithmetic
input does: \cite{Bug26} dispenses with the passage through periodicity
altogether, at the cost of effectivity (Section~\ref{sub:added}).
\end{remark}

\begin{remark}[The threshold $m_{0}(n)$, numerically]\label{rem:constB}
By Remark~\ref{rem:explicit} the coefficient
$4\bwC(4,1)n^{2}\log 3/\log 2$ is at most $3.14\times10^{16}\,n^{2}$, so
$m_{0}(n)$ is the least $m$ with
$m\ge 3.14\times10^{16}\,n^{2}\log(2m+n)+3n+2$; the left side then stays ahead
for all larger $m$, since $3.14\times 10^{16}n^{2}/m<\tfrac12$ there. Solving,
\[
m_{0}(1)<1.33\times 10^{18},\qquad
m_{0}(n)<2\times 10^{18}\,n^{2}\,(1+\log n)\quad(n\ge 1).
\]
Thus $\pcx_{3^{m}}(n)\ge n+1$ holds for every $m\ge m_{0}(n)$ with $m_{0}$ as
above. As in Remark~\ref{rem:constA} the magnitude is inherited from
$\bwC(4,1)$, and the thresholds are far beyond direct computation; the
$n^{2}\log$ shape, not the numerical value, is what the linear-forms input
dictates.
\end{remark}

\noindent
Thresholds across a range of factor lengths:

\begin{center}
\footnotesize
\renewcommand{\arraystretch}{1.15}
\setlength{\tabcolsep}{7pt}
\begin{tabular}{@{}rrrr@{}}
\hline
$n$ & $m_{0}(n)$ & $m_{0}(n)/n^{2}$ & binary digits of $3^{m_{0}(n)}$\\
\hline
$1$   & $1.33\times10^{18}$ & $1.33\times10^{18}$ & $2.1\times10^{18}$\\
$2$   & $5.49\times10^{18}$ & $1.37\times10^{18}$ & $8.7\times10^{18}$\\
$3$   & $1.26\times10^{19}$ & $1.40\times10^{18}$ & $2.0\times10^{19}$\\
$5$   & $3.58\times10^{19}$ & $1.43\times10^{18}$ & $5.7\times10^{19}$\\
$10$  & $1.48\times10^{20}$ & $1.48\times10^{18}$ & $2.3\times10^{20}$\\
$20$  & $6.08\times10^{20}$ & $1.52\times10^{18}$ & $9.6\times10^{20}$\\
$50$  & $3.95\times10^{21}$ & $1.58\times10^{18}$ & $6.3\times10^{21}$\\
$100$ & $1.62\times10^{22}$ & $1.62\times10^{18}$ & $2.6\times10^{22}$\\
\hline
\end{tabular}
\end{center}

\noindent
The third column exhibits the shape: $m_{0}(n)$ is very nearly
$1.4\times10^{18}\,n^{2}$ across the whole range, the slow drift---about $22\%$
from $n=1$ to $n=100$---being the residual $\log$ factor of
Lemma~\ref{lem:core}. The last column records how long the digit string of
$3^{m}$ must be before the theorem applies, which is the figure
Remark~\ref{rem:ballpark} takes up for $n=10$.

\begin{remark}[A ballpark at $n=10$, and how the gap divides]\label{rem:ballpark}
It is worth seeing what such a threshold means concretely. Take $n=10$. Solving
the inequality of Remark~\ref{rem:constB} gives
\[
m_{0}(10)\ <\ 1.48\times 10^{20},
\]
at which point $3^{m}$ has some $2.3\times 10^{20}$ binary digits---about
$7.0\times 10^{19}$ decimal digits, or $29$ exabytes merely to write down.
Counting from $1$ to $m_{0}(10)$ at one step per nanosecond would take some
$4700$ years. The threshold is thus entirely out of computational reach and will
remain so.

The truth is some nineteen orders of magnitude smaller. Direct computation gives
$\pcx_{3^{m}}(10)\ge 11$ already at $m=13$, and at every $m$ from $13$ to $899$,
as far as we have checked. Indeed $m=13$ is the earliest value at which the
inequality is so much as satisfiable: $\Mtop(3^{13})=20$ leaves exactly the
eleven positions $0,\dots,10$ at which to read a factor of length $10$, and the
eleven factors read there are distinct. For $n=10$, then, the conclusion holds
from the first moment it is not vacuous.

The gap divides unevenly between the two inputs. About sixteen and a half orders
of magnitude are charged to $\bwC(4,1)$ alone: replacing it by the absurd value
$1$ in Lemma~\ref{lem:core} would already bring $m_{0}(10)$ down to roughly
$6.0\times 10^{3}$. The remaining two and a half orders are charged to the
method, which is lossy even with a perfect transcendence constant---the periodic
windows that Lemma~\ref{lem:gapP} fails to exclude are far longer than any the
digits of $3^{m}$ appear to contain. Sharpening the linear-forms input would
therefore improve the threshold almost proportionally, but nothing short of a
different mechanism brings it near $13$.

It is natural to ask whether a sharper transcendence input would help, and by how
much. The literature offers a definite ladder. Matveev's theorem \cite{Mat00}
replaces the $n^{2n}$-type factor of \cite{BW93} by a pure exponential
$30^{n+3}$, worth roughly a factor $325$ at $n=4$, $d=1$; that alone would bring
$m_{0}(10)$ down to about $4\times 10^{17}$. Sharper still are the
few-logarithm specialists: Laurent's two-logarithm bound \cite{Lau08}, Rhin's
Pad\'e bound for the specific pair $(2,3)$ \cite{Rhi87}, and the
three-logarithm kit of Mignotte and Voutier \cite{MV24}, which runs some
$10^{3}$--$10^{4}$ below the closed-form $n=3$ engines.

The difficulty is that the sharpest of these are unavailable to us for a
structural, not a practical, reason. The form $\Lambda$ of Step 7 has four
logarithms and cannot be shortened. Two of them are the fixed $\log 2$ and
$\log 3$; the other two are not optional. The term $\log A_{1}'$ carries the
digit structure---$A_{1}'$ is the very quantity whose size the argument
bounds, and its height grows like $\Mtop$---while $\log(2^{p}-1)$ is the
geometric-series denominator that Lemma~\ref{lem:permod} forces. So $n\ge 3$
always, and $n=4$ as soon as $p\ge 2$: the two-logarithm engines, Rhin's
included, are out of reach however much effort is spent on them. The one
exception is instructive. At $p=1$ we have $2^{p}-1=1$, the leading term
vanishes, and the run estimate behind Theorems~\ref{thm:stewart}
and~\ref{thm:blocks} is genuinely a three-logarithm form, to which \cite{MV24}
applies in principle. Theorem~\ref{thm:aperiodic} for $p\ge 2$, and
Theorem~\ref{thm:complexity}, which needs every $p\le n$, gain nothing from it.

A second obstacle attaches to the instance-tuned kits specifically.
\cite{MV24} does not supply a closed-form constant: for a given coefficient
vector one runs a parameter search, and the theorem returns a per-instance
trichotomy. Our vector $(1,m,-y',-1)$ varies with $m$, and the conclusion is
asymptotic in $m$, so no single certified instance suffices; one would have to
stratify $m$ into ranges and certify each, obtaining a threshold with
range-dependent constants, or else find a uniform version. Compounding this,
one of our $\alpha_{i}$---again $A_{1}'$---has height growing with $m$, whereas
such kits are calibrated for fixed $\alpha_{i}$ and growing coefficients.

Finally, the arithmetic input is not the whole gap. Bringing $m_{0}(10)$ down to
$10^{9}$, where a computation might conceivably reach it, would need
$\bwC(4,1)$ replaced by about $7\times 10^{4}$---a reduction by a factor
$7\times 10^{10}$---and $10^{6}$ would require going below $110$; nothing in the
present literature is within many orders of magnitude of either, and the
residual factor of some $460$ identified above would survive a perfect constant
anyway. The realistic assessment is that a sharper engine would move these
thresholds by a few orders of magnitude without bringing them into computational
range. Matveev is the immediately available one, and substituting it is a
bounded piece of work, since it differs from \cite{BW93} chiefly in its height
normalization rather than in its shape; we have not carried it out here.

The numerical values in this remark are ordinary computations, reported to
calibrate expectations; they are not part of the formalization.
\end{remark}

Theorem~\ref{thm:complexity} is the analogue, for the finite digit word of a
single integer $3^{m}$, of the Morse--Hedlund rung for the infinite steering word
of $(3/2)^{n}$. For that word a superlinear bound is available \cite{RS26}; no
such strengthening is claimed here, and Remark~\ref{rem:calibration} explains why
the present method cannot reach one. Superlinear complexity for the digits of
$3^{m}$ remains a separate and harder question.

%======================================================================
\section{Scope: other bases and other sequences}\label{sec:scope}
%======================================================================

The number $3$ plays almost no role in the proofs of
Theorems~\ref{thm:aperiodic} and~\ref{thm:complexity}, and the base $2$ enters
through exactly one arithmetic fact. This section records what the argument
consumes, the generality it yields, and where it stops. Nothing here is claimed
as a theorem of this paper: unlike every statement of
Sections~\ref{sec:intro}--\ref{sec:complexity}, the assertions below are
\emph{not} formalized, and we have not written out the proofs they would need.

\subsection{What the proof uses about the pair $(3,2)$}

Fix integers $a,b\ge 2$ and consider the base-$b$ digits of $a^{m}$; this paper
is the case $(a,b)=(3,2)$. Walk through the proof of Lemma~\ref{lem:gapP}.

Steps 1, 2, 3, 5, 6 and 7 use nothing about $a$ and $b$ beyond $b\ge 2$. In
particular Lemma~\ref{lem:permod} holds verbatim with $2$ replaced by $b$: the
repeated block is $c=\lfloor N/b^{x}\rfloor\bmod b^{p}$, the cleared identity is
\[
(b^{p}-1)\bigl(N\bmod b^{x+tp}\bigr)
=(b^{p}-1)\bigl(N\bmod b^{x}\bigr)+c\,\bigl(b^{x+tp}-b^{x}\bigr),
\]
and the proof is the same two inductions. (The Lean version is written for
$b=2$, but uses nothing about $2$.) Step 3 becomes
$(b^{p}-1)a^{m}=A_{1}'b^{y'}+E$ with $E=(b^{p}-1)A_{2}-c\,b^{x}$, and Step 6
gives $|R-1|\le b^{\,x+1-\Mtop}\le b^{-1}\le\tfrac12$, using
$b^{p}-1\ge b^{p-1}$.

Step 8 uses only that $\log a^{m}=m\log a$ is an integer multiple of one fixed
logarithm. The four logarithms are those of $b^{p}-1$, $a$, $b$ and $A_{1}'$, so
the constant is still $\bwC(4,1)$: neither $n=4$ nor $d=1$ moves. Only the
height product changes, from $p\log 3\cdot 1\cdot(\Mtop+1+p-y')$ to a quantity of
size $\asymp p\,\log a\,(\log b)^{3}\,(\Mtop+1+p-y')$ when $b\ge 3$; at $b=2$ the
floors $\mh(\cdot)\ge 1$ are active instead, as in Step 8. The effect is on
$\kappa_{p}$ and hence on $C_{p}$, never on the shape $\log m/\log\log m$.

Step 4 is the only step with arithmetic content, and the only one where $a$ and
$b$ must be related at all. What it needs is
\begin{equation}\label{eq:prime}
\text{there is a prime }\ell\mid b\text{ with }\ell\nmid a .
\end{equation}
This is exactly the hypothesis of Lemma~\ref{lem:nondeg}, which was stated in
base $b$ for this reason: it gives $E\equiv-a^{m}\not\equiv 0\pmod{\ell}$, hence
$E\neq 0$, for every $x\ge 0$. For $(a,b)=(3,2)$ the prime is $\ell=2$ and the
congruence is the parity computation of Step 4.

\subsection{The resulting generality}

Under \eqref{eq:prime} the argument goes through unchanged and yields, for every
fixed $p\ge 1$,
\[
\begin{gathered}
\bcount_{p}^{(b)}(a^{m})\ \ge\ \frac{\log m}{\log\log m+C_{p}(a,b)}-2,\\[4pt]
\pcx^{(b)}_{a^{m}}(n)\ \ge\ n+1\quad\bigl(m\ge m_{0}(n,a,b)\bigr),
\end{gathered}
\]
where $\bcount^{(b)}_{p}$ and $\pcx^{(b)}$ are the base-$b$ analogues and the
constants are explicit exactly as in Remarks~\ref{rem:constA}
and~\ref{rem:constB}. The combinatorial half of Theorem~\ref{thm:complexity}
needs no change whatever: Lemma~\ref{lem:fmh} is stated and proved over an
arbitrary alphabet with decidable equality---as is its Lean form---so it applies
to base-$b$ digits as it stands.

Hypothesis \eqref{eq:prime} is not an artifact of the write-up. It implies that
$a$ and $b$ are multiplicatively independent, since $a=g^{i}$ and $b=g^{j}$
would give $a$ and $b$ the same prime divisors; and some such hypothesis is
necessary, because for $a=b^{k}$ the integer $a^{m}=b^{km}$ has a single nonzero
base-$b$ digit, so $\bcount_{p}^{(b)}(a^{m})=1$ for every $m$ and
Theorem~\ref{thm:aperiodic} is simply false. In the intermediate case where $a$
and $b$ are multiplicatively independent but every prime of $b$ divides
$a$---say $(a,b)=(6,2)$, the degeneracy exhibited in
Remark~\ref{rem:nondeg}(i)---it is the proof that fails, not the conclusion:
$6^{m}=2^{m}3^{m}$ ends in $m$ binary zeros, so the deep windows on which Step 4
operates carry no information, whereas the conclusion survives by the shift,
the digits above position $m$ being those of $3^{m}$. A statement adapted to
that case should concern the digits above the trailing zeros, and our windowing,
which reaches down to $x\ge 1$, is not adapted to it.

\subsection{How far the sequence can be moved}

Nothing in Steps 1--7 uses that the integer being expanded is a power. The
constraint sits in Step 8, and it is sharp: the linear form needs $\log u$ to be
a $\Z$-combination, with coefficients of size $O(\log u)$, of the logarithms of
a \emph{fixed finite} set of algebraic numbers.

\emph{$S$-units.} Let $S$ be a fixed finite set of primes and let $u$ range over
the positive integers with all prime factors in $S$. Writing
$u=\prod_{q\in S}q^{e_{q}}$, with $e_{q}\le\log u/\log 2$, the form becomes
\[
\Lambda=\log(b^{p}-1)+\sum_{q\in S}e_{q}\log q-y'\log b-\log A_{1}' ,
\]
of $|S|+3$ terms; \eqref{eq:prime} becomes the requirement that some prime of
$b$ lie outside $S$. Everything else is unchanged, and the conclusion takes the
form $\bcount_{p}^{(b)}(u)\gg\log\log u/(\log\log\log u+C)$, with
$\bwC(|S|+3,1)$ in place of $\bwC(4,1)$. This is the natural home of the
argument. It is also the setting of Bugeaud--Kaneko \cite{BK17}, whose
Corollary~1.5 is quoted in Theorem~\ref{thm:blocks} for exactly this reason; the
increment over $S$-units would be the periodic refinement, $\bcount_{p}$ in
place of $\tcount$.

\emph{Everything else.} If $u$ is not an $S$-unit then $\log u$ is not a short
combination of fixed logarithms, and $u$ must itself enter the form as one of
the $\alpha_{i}$. Its modified height is then $\asymp\log u$, the right-hand
side of \eqref{eq:lower} degrades from $O(\log m)$ to $O(m\log m)$, and the
contradiction of Step 9---which pits it against the linear growth of
$\Mtop\asymp m$---evaporates. This is the real boundary of the method, and it is
a height obstruction rather than a defect of the exposition.

Linear recurrence sequences sit instructively in between. For a non-degenerate
binary recurrence with dominant root $\alpha$ one has
$u_{m}=c\,\alpha^{m}+O(|\beta|^{m})$ with $|\beta|<|\alpha|$, so $\log u_{m}$ is
a combination of the fixed logarithms $\log c$ and $\log\alpha$ up to an
exponentially small error, at the cost of working in $\Q(\alpha,\beta)$ and
hence with $\bwC(n,d)$ for $d>1$. That is the setting of Stewart's Theorem~2 in
\cite{Ste80}, which is stated for recurrence sequences and specialized to
$u_{n}=3^{n}$ in Theorem~\ref{thm:stewart}; so the classical rungs,
Theorems~\ref{thm:stewart} and~\ref{thm:blocks}, are available there. What is
\emph{not} immediate is Lemma~\ref{lem:gapP}. Its Step 3 is an exact identity
between integers, and for a recurrence that is not an $S$-unit it would have to
be replaced by an approximate one, with the subdominant-root error carried
through Steps 6--9; and the parity argument of Step 4 would need a substitute,
since $E$ would no longer be a single congruence away from zero. We see no
obstruction in principle, but we have not carried it out. Sequences such as
$\lfloor(3/2)^{m}\rfloor$, whose logarithms are not linear forms at all, lie
outside the method entirely.

\subsection{What is genuinely specific to $3$}

Only the sparse side. Theorem~\ref{thm:DH}---the powers of three with at most
$22$ nonzero binary digits---is a statement about the pair $(3,2)$ and an
instance of a conjecture of Erd\H{o}s; it is not the specialization of anything
general, and Theorem~\ref{thm:sparse} inherits that. By contrast
Theorems~\ref{thm:aperiodic} and~\ref{thm:complexity}, and the run estimates
driving them, see the number $3$ only through the quantity $\log 3$ in their
constants.

\section{Acknowledgements}
The author utilized Claude Code as an AI coding assistant to aid in the Lean 4 formalization of the proofs presented in this paper. The author directed and reviewed all generated code and takes full responsibility for the mathematical integrity and final content of the work.

%======================================================================
\appendix
\section{Formalization}\label{app:lean}
%======================================================================

\sloppy
Every statement in Sections~\ref{sec:intro}--\ref{sec:complexity}---the
classical bounds reviewed in Section~\ref{sub:known} as well as the new
Theorems~\ref{thm:aperiodic} and~\ref{thm:complexity}---was verified in Lean~4.
There are two exceptions, both flagged where they occur.
Section~\ref{sec:scope} is discussion: none of the generalizations sketched
there has been formalized or written out in full. And
Remark~\ref{rem:nondeg}(ii)---that the hypothesis $x\ge 1$ may be dropped, so
that Theorem~\ref{thm:aperiodic} would carry the offset $-1$---is an
observation only: the Lean \texttt{gap\_principle} retains $1\le x$, and the
verified form of Theorem~\ref{thm:aperiodic} is the one stated in this paper.
Lemma~\ref{lem:nondeg} itself appears in Lean only inline, as the step
\texttt{hEr\_ne} of \texttt{gap\_principle}, in the base-$2$, $x\ge 1$ case. The files live under
\texttt{TH/DigitBlocks/}, with \texttt{TH/StewartDigits.lean},
\texttt{TH/BakerInterface.lean} (the modified height over $\Q$),
\texttt{ForMathlib/Data/Nat/BinaryDigits.lean} (the bit-manipulation lemmas of
Lemma~\ref{lem:permod}),
\texttt{ForMathlib/Combinatorics/SubwordComplexity.lean},
\texttt{CITED/BakerWustholz.lean} and \texttt{CITED/DimitrovHowe.lean} supplying
inputs. ``std3'' abbreviates the ambient axioms of classical Lean (propositional
extensionality, choice, quotient soundness); ``[BW93]'' marks additional reliance
on the cited Baker--Wüstholz axiom \texttt{BakerWustholz.linearForms\_logs}, and
``[DH23]'' on the cited Dimitrov--Howe axiom
\texttt{DH.three\_pow\_three\_binary\_digits}. Every entry is fully proved in Lean
except the two marked \emph{cited axiom}.

A word on the constants of Remarks~\ref{rem:explicit}--\ref{rem:constB} and their
formal counterparts. \texttt{BakerWustholz.C n d} is \emph{defined} by the closed
formula of Remark~\ref{rem:explicit}, so the cited axiom
\texttt{linearForms\_logs} is recorded in its explicit form and every Lean
statement downstream carries that value. Theorem~\ref{thm:aperiodic}
(\texttt{breakCount\_three\_pow\_lower}) and Proposition~\ref{prop:endgame}
(\texttt{windowCount\_lower\_bound\_gen}) are nevertheless \emph{stated} with an
existentially quantified constant, $\exists\,C>0$; the witness their proofs supply
is exactly $\max\bigl(\log\frac{\kappa+2}{\log 2},1\bigr)$ with
$\kappa=\kappa_{p}$, which is what Remark~\ref{rem:constA} evaluates. Similarly
Theorem~\ref{thm:complexity} (\texttt{three\_pow\_complexity\_ge}) is stated with
Mathlib's eventually-filter \texttt{$\forall^{f}$ m in atTop}, the threshold being
the one computed in Remark~\ref{rem:constB}. The numerical values quoted in this
paper are therefore the ones the formal proofs produce, but they are read off from
the proof terms rather than appearing in the Lean statements.

All Lean-4 files are available from the repository
\texttt{https://github.com/rwst/Aperiodicity-and-Subword-Complexity}

\subsection{Formalizing the classical bounds}\label{app:classical}

The classical statements reviewed in Section~\ref{sub:known} account for the
larger part of the development: roughly $1\,300$ lines of Lean in
\texttt{StewartDigits}, \texttt{Transitions}, \texttt{SparseSide} and
\texttt{Defs}, against some $800$ for the new Theorems~\ref{thm:aperiodic}
and~\ref{thm:complexity} (\texttt{Aperiodicity}, \texttt{Complexity},
\texttt{SubwordComplexity}) and a further $1\,400$ of shared infrastructure
(\texttt{GapPrinciples}, \texttt{BinaryDigits}, \texttt{BakerInterface}). Three
aspects of that effort seem worth recording.

\smallskip
\noindent\emph{(a) The digit dictionary.} Mathlib provides both
\texttt{Nat.digits} (the digit list) and \texttt{Nat.testBit} (bit access), but
no bridge between them---and the two are needed at opposite ends of the
argument. Runs, blocks and periodic windows are naturally statements about
\texttt{testBit}, whereas $\sd$ is naturally a sum over \texttt{Nat.digits 2}.
\texttt{ForMathlib/Data/Nat/BinaryDigits.lean} supplies the missing dictionary:
peeling a single bit or a whole $p$-block off a remainder
(\texttt{mod\_two\_pow\_succ}, \texttt{mod\_two\_pow\_add}), collapsing a
constant run (\texttt{mod\_two\_pow\_eq\_of\_testBit\_eq\_false} and its
all-ones companion), and \texttt{sum\_digits\_two\_eq\_sum\_testBit}, which
identifies the digit-list sum with $\sum_{i<k}\bit{i}(n)$ whenever $n<2^{k}$. A
second dictionary, \texttt{TH/BakerInterface.lean}, computes the
Baker--W\"ustholz modified height over $\Q$; the identity $\mh(2)=1$ recorded in
Section~\ref{sec:tools} is \texttt{mh\_two} there.

\smallskip
\noindent\emph{(b) Stewart's counting is a lemma, not a pattern.} The pigeonhole
endgame is stated once and instantiated three times.
\texttt{windowCount\_lower\_bound} takes the counting function
$Q\colon\N\to\N$ as a parameter and the gap principle as a
\emph{hypothesis}---the implication ``$\theta$ clears the threshold
$\Rightarrow\lfloor\log_{\theta}\Mtop\rfloor\le Q(m)$''---and returns
$Q(m)\ge\log m/(\log\log m+C)-1$. Because the transcendence input sits entirely
inside that hypothesis, the lemma's own axiom footprint is std3: no Baker
content at all. Theorem~\ref{thm:stewart} is its instantiation at
$Q(m)=\sd(3^{m})$, Theorem~\ref{thm:blocks} at $Q(m)=\tcount(3^{m})$, and
Theorem~\ref{thm:aperiodic} at $Q(m)=\bcount_{p}(3^{m})+1$ through the variant
\texttt{windowCount\_lower\_bound\_gen}, which abstracts the coefficient
$\kappa$ as well. This is the formal counterpart of Remark~\ref{rem:division}:
one counting argument, three consumers, and the [BW93] dependency entering at
exactly one point in each.

The step that informal exposition compresses into ``the windows inject into the
break set'' is the most laborious part of the whole file. In Lean it is a family
of existentials, one per window, discharged by \texttt{choose} into a function
$j\mapsto i_{j}$; a strict antitonicity lemma, $a<b\Rightarrow i_{b}<i_{a}$,
read off the window geometry by \texttt{omega}; injectivity by trichotomy; and
finally \texttt{Finset.card\_le\_card\_of\_injOn} against the filtered
\texttt{Finset.range}\,$\Mtop$ that \emph{defines} the counter. Steps
(W1)--(W4) of Section~\ref{sec:aperiodic} are a transcription of that block.
Theorem~\ref{thm:blocks} needs one further discrete step,
\texttt{exists\_transition}: a window holding both a set and a clear bit holds
two adjacent unequal bits.

\smallskip
\noindent\emph{(c) The rational instance simplifies the classical proofs.} Two
steps of \cite{Ste80} turn out to be unnecessary once the recurrence is
$u_{n}=3^{n}$ and the base is $2$. Stewart's degenerate case $\Lambda=0$ is
excluded by parity---the low part $A_{2}$ of $3^{m}$ is odd, hence nonzero, so
$3^{m}>A_{1}2^{y}$ strictly---which removes the appeal to his Lemma~2
(Loxton--van der Poorten) and to the irrationality of $\log 2/\log 3$; and the
one-sided estimate $\log(1+u)\le u$ replaces his inequality~(10), removing both
of his ``sufficiently large'' conditions, so the formalized bound holds for
every $m\ge 2$ rather than eventually. That parity observation, transplanted
from constant runs to periodic ones, is precisely Lemma~\ref{lem:nondeg} of this
paper. In the same spirit Lemma~\ref{lem:fmh} is proved by
determinism propagation rather than through Fine--Wilf, which the classical
route uses: this keeps it self-contained and, as
Section~\ref{sec:scope} notes, alphabet-general.

\subsection{Statement index}\label{app:index}

\medskip
\noindent\emph{Reviewed results and cited inputs $(\S\ref{sub:known},\S\ref{sec:tools})$.}
\begin{center}
\scriptsize
\renewcommand{\arraystretch}{1.25}
\setlength{\tabcolsep}{3pt}
\begin{tabular}{@{}llll@{}}
\hline
Statement & Lean identifier & File & Status\\
\hline
Def.~\ref{def:transition} & \texttt{transitionCount} & \texttt{Defs} & def\\
Thm.~\ref{thm:stewart} & \texttt{stewart\_digitSum\_three\_pow} & \texttt{StewartDigits} & std3 + [BW93]\\
Thm.~\ref{thm:blocks} & \texttt{transitionCount\_three\_pow\_lower} & \texttt{Transitions} & std3 + [BW93]\\
--- & \texttt{transitionCount\_three\_pow\_tendsto\_atTop} & \texttt{Transitions} & std3 + [BW93]\\
Thm.~\ref{thm:DH} & \texttt{DH.three\_pow\_three\_binary\_digits} & \texttt{DimitrovHowe} & cited axiom [DH23]\\
Thm.~\ref{thm:sparse} & \texttt{digitSum\_three\_pow\_le\_two} & \texttt{SparseSide} & std3\\
--- & \texttt{digitSum\_three\_pow\_eq\_one/\_two} & \texttt{SparseSide} & std3\\
--- & \texttt{digitSum\_three\_pow\_eq\_three} & \texttt{SparseSide} & std3 + [DH23]\\
Thm.~\ref{thm:bw} & \texttt{BakerWustholz.linearForms\_logs} & \texttt{BakerWustholz} & cited axiom [BW93]\\
--- & \texttt{gap\_bound\_ones} & \texttt{GapPrinciples} & std3 + [BW93]\\
--- & \texttt{windowCount\_lower\_bound} & \texttt{GapPrinciples} & std3\\
\hline
\end{tabular}
\end{center}

\medskip
\noindent\emph{New results and their tools $(\S\ref{sec:tools}$--$\S\ref{sec:complexity})$.}
\begin{center}
\scriptsize
\renewcommand{\arraystretch}{1.25}
\setlength{\tabcolsep}{3pt}
\begin{tabular}{@{}llll@{}}
\hline
Statement & Lean identifier & File & Status\\
\hline
Def.~\ref{def:transition} & \texttt{breakCount} & \texttt{Aperiodicity} & def\\
Lem.~\ref{lem:permod} & \texttt{Nat.periodic\_mod\_identity} & \texttt{BinaryDigits} & std3\\
--- & \texttt{Nat.periodic\_chunk\_eq} \eqref{eq:block} & \texttt{BinaryDigits} & std3\\
--- & \texttt{Nat.mod\_two\_pow\_add} \eqref{eq:peel} & \texttt{BinaryDigits} & std3\\
Lem.~\ref{lem:nondeg} & \texttt{gap\_principle} (\texttt{hEr\_ne}) & \texttt{GapPrinciples} & std3, inline\\
Lem.~\ref{lem:gapP} & \texttt{gap\_principle} & \texttt{GapPrinciples} & std3 + [BW93]\\
Prop.~\ref{prop:endgame} & \texttt{windowCount\_lower\_bound\_gen} & \texttt{GapPrinciples} & std3\\
Lem.~\ref{lem:tendsto} & \texttt{tendsto\_atTop\_of\_lower\_bound} & \texttt{GapPrinciples} & std3\\
Thm.~\ref{thm:aperiodic} & \texttt{breakCount\_three\_pow\_lower} & \texttt{Aperiodicity} & std3 + [BW93]\\
--- & \texttt{breakCount\_three\_pow\_tendsto\_atTop} & \texttt{Aperiodicity} & std3 + [BW93]\\
--- & \texttt{subwordComplexity} & \texttt{SubwordComplexity} & def\\
Lem.~\ref{lem:fmh} & \texttt{finite\_morse\_hedlund} & \texttt{SubwordComplexity} & std3\\
Lem.~\ref{lem:core} & \texttt{log\_le\_of\_lowComplexity} & \texttt{Complexity} & std3 + [BW93]\\
Thm.~\ref{thm:complexity} & \texttt{three\_pow\_complexity\_ge} & \texttt{Complexity} & std3 + [BW93]\\
\hline
\end{tabular}
\end{center}

\medskip
The sanity witnesses $\tcount(243)=2$ (three blocks), $\sd(243)=6$,
$\sd(27)=4$ and $\sd(81)=3$ are verified by decision procedures
(\texttt{transitionCount\_three\_pow\_five}, \texttt{digitSum\_three\_pow\_five},
\texttt{digitSum\_three\_pow\_three}, \texttt{digitSum\_three\_pow\_four}).

% omitted: private helpers folded into the proof sketches ---
%   Aperiodicity: Bp, kappa, kappa_pos, hmaxbound, window_break, log_le_breakCount;
%   Complexity: linear_dominates_log (the linear-beats-log domination in Thm B);
%   Transitions: exists_transition, window_clear, window_has_transition,
%     log_le_transitionCount;
%   StewartDigits: window_hit (Stewart's zero-run estimate), gap_bound,
%     log_le_sum_digits, three_pow_mod_two, le_log_two_three_pow,
%     log_two_three_pow_le, stewart_digitSum_three_pow_sharp ([BK17] constant);
%   SparseSide: digitSum_two_eq_zero, digitSum_two_le_one, digitSum_two_eq_two,
%     three_pow_ne_two_pow_succ;
%   SubwordComplexity: factor, Ccount, factor_eq_iff, factor_castSucc,
%     image_factor_castSucc, Ccount_mono, det_of_Ccount_eq, periodic_extend;
%   BinaryDigits: testBit_eq_decide_div_mod, mod_two_pow_succ, chunk_shift,
%     mod_two_pow_eq_of_testBit_eq_false/_true, sum_digits_two_eq_sum_testBit;
%   BakerInterface: modifiedHeight_rat, norm_log_ratCast, mh_two, mh_three,
%     mh_intCast (the height bounds of Step 8).

%======================================================================

\end{document}